\UseRawInputEncoding 
\documentclass[11pt]{amsart}
\usepackage{pdfsync}
\usepackage{tikz}
\usepackage{amsmath,amssymb,amscd,amsfonts,verbatim}
\usepackage{paralist}
\usepackage[mathscr]{eucal}

\usepackage{enumerate}

\usepackage[pdftex]{hyperref}
\hypersetup{citecolor=blue,linktocpage}

\newtheorem{thm}{Theorem}[section]
\newtheorem{lem}[thm]{Lemma}

\newtheorem{prop}[thm]{Proposition}
\newtheorem{cor}[thm]{Corollary}

\newtheorem{assu-nota}[thm]{Assumption--Notation}
\theoremstyle{definition}
\newtheorem{ex}[thm]{Example}
\newtheorem{remark}[thm]{Remark}
\newtheorem{defn}[thm]{Definition}

\newcommand{\C}{\mathbb C}

\newcommand{\pp}{\mathbb P}

\DeclareMathOperator{\Pic}{Pic}

\newcommand{\OO}{\mathcal{O}}

\usepackage[colorinlistoftodos]{todonotes}

\numberwithin{equation}{section}
\title{Numerical properties of exceptional divisors of birational morphisms of smooth surfaces}
\author{Vicente Lorenzo \textsuperscript{1}} 

\author{ Margarida Mendes Lopes \textsuperscript{2}} 
\author{ Rita Pardini \textsuperscript{3}}

\begin{document}
 \begin{abstract}  
We make a very detailed analysis of effective divisors whose support is contained in the exceptional locus of a birational morphism of smooth projective surfaces. As an application we extend  Miyaoka's inequality on the number of canonical singularities on a  projective normal surface with non-negative Kodaira dimension to the non minimal case, obtaining a slightly better result than known extensions by Megyesi and Langer.


 
 \end{abstract}
 \maketitle

 \footnotetext[1]{Partially supported by FCT/Portugal through the program Lisbon Mathematics PhD, scholarship FCT - PD/BD/128421/2017}
  \footnotetext[2]{Partially supported by  FCT/Portugal  through Centro de An\'alise Matem\'atica, Geometria e Sistemas Din\^amicos (CAMGSD), IST-ID, projects UIDB/04459/2020 and UIDP/04459/2020}
   \footnotetext[3]{Partially supported by the project PRIN 2017SSNZAW$\_$004 ``Moduli Theory and Birational Classification"  of Italian MIUR and a member of GNSAGA of INDAM}
 \setcounter{tocdepth}{1}
\tableofcontents

 \section{Introduction}
 
 It is well known  that any  birational morphism $p\colon T\to W$  between two non singular projective surfaces is a composition of blow-ups at smooth points.   
 In spite of this simple description, the exceptional divisor of $p$ can be quite intricate if  some of the blow-ups occur at infinitely near points, i.e.,  points lying in the exceptional divisor of one of the previous blow-ups. This paper was born because in dealing with other problems we realized that the properties of divisors whose support  consists of $p$-exceptional curves are scattered across the literature and sometimes hard to locate.
  
  So here we make a very detailed analysis of these divisors.   As an application, we extend  Miyaoka's inequality on the number of canonical singularities on a  projective normal minimal surface  with  non-negative Kodaira dimension to the non minimal case.    Let us point out that  extensions to the non minimal case appear already  as a particular case  of  results of G. Megyesi and A. Langer in a more general framework (\cite{megyesi}, \cite{langer1}, \cite{langer}). Our proof uses only Miyaoka's result and  the inequality we obtain is slightly better than Megyesi's and Langer's (see Remark \ref{langer}).

         After this paper was written,  we learnt that the  exceptional divisors we are considering fit into the framework of ``sandwiched singularities'',  studied  in the substantial paper \cite{spiva} by  M.~Spivakovsky. Therefore some of the  properties we establish can be found there.  However  we believe that the present  note, having these properties in an explicit and self contained form,  can be a useful reference.  
   \medskip

The paper is organized as follows:  in \S 2 we quickly review the notions of arithmetical genus and numerical connectedness of curves on a smooth surface;  in \S 3 we discuss  properties of (reducible) $-1$-curves on a smooth surface $T$.   In \S 4 and \S 5  we establish many properties of curves  $\Delta$ with  $\Delta^2=-2$ and $K_T\Delta=0$  such that  there is  a birational morphism $p\colon T\to W$ to a smooth surface that contracts $\Delta$ to  a point (note in particular Propositions \ref{thm: An}  and  \ref{prop: nef}).  Finally, in  \S 6 we apply these results  to prove the extension of Miyaoka's inequality (Theorem \ref{thm: ineq}). 
\smallskip

\noindent {\em Acknowledgements:} we are indebted to the referee for useful suggestions, in particular for pointing out the references \cite{hi} and  \cite{langer}.

\subsection  {Notation} 
A {\em surface} $S$ is a normal  complex projective surface; when $S$ is smooth or has canonical singularities we denote as usual by $K_S$ the canonical class and write $p_g(S):=h^0(K_S)$ (\!{\em geometric genus})  and $q(S):=h^1(\OO_S)$ (\!{\em irregularity}). 

A   {\it curve} $D$ on a smooth surface $S$ is  a non-zero effective divisor (in particular $D$ is Gorenstein).
Given a curve $D$, we denote  its dualizing sheaf  by $\omega_D$  and  its arithmetic genus $p_a(D)$,  where by definition $1-p_a(D)=h^0(D,\OO_D)-h^1(D,\OO_D)$.
 A curve  $D$ is {\em $m$-connected} if for every decomposition $D=A+B$ as the sum of two curves $A$ and $B$ one has $A B \geq m$.
 A {\em $-2$-curve} (or {\em nodal curve}) on a smooth surface $S$ is an irreducible curve that is isomorphic to $\pp^1$ and has self-intersection $-2$.

 For a rational number $p$, $\lfloor p\rfloor$  is the largest integer $\leq p$.

\section{Arithmetic genus and $m-$connectedness} 
 In this section we collect some well known facts on curves on a surface that will be used throughout the paper.

 \noindent  Let $D$ be a curve on a surface $S$: 

\begin{itemize} 
\item If $D=D_1+D_2$, with $D_1$ and $D_2$ curves, then there is an exact sequence (\!{\em decomposition sequence}):
$$0\to \OO_{D_1}(-D_2)\to\OO_D\to \OO_{D_2}\to 0.$$
\item $\omega_D=(K_S+D)|_D$ and  $2p_a(D)-2= K_SD+D^2$ ({\em adjunction formula}).

\item By Serre duality,  one has $h^0(D,\OO_D)=h^1(D,\omega_D)$ and $h^1(D,\OO_D)=h^0(D,\omega_D)$. 
\item If $D$ decomposes as the sum of two curves $D_1,D_2$ then
$2p_a(D)-2=K_SD+D^2=K_SD_1+K_SD_2+D_1^2+D_2^2+2D_1D_2=2p_a(D_1)-2+2p_a(D_2)-2+2D_1D_2$
and so $p_a(D)=p_a(D_1)+p_a(D_2)-1+D_1D_2$.

\item (\cite[Lem.~A.4]{cfm}) Let $D$ be an
$m$-connected curve on the surface $S$ and let $D=D_1+D_2$ with $D_1$, $D_2$
curves. Then:
\begin{enumerate}[(\rm i)]
\item if  $D_1  D_2=m$, then $D_1$ and $D_2$ are
$\lfloor\frac{m+1}{2}\rfloor$-connected;

\item if $D_1$ is chosen to be minimal subject to the
condition $D_1  (D-D_1) = m$, then $D_1$ is $\lfloor\frac {m+3}{2}\rfloor$-connected.
\end{enumerate} 
\end{itemize}

\begin{lem}\label{dec}\cite[Lem.~1.2]{k} Let $D$ be a curve  such that $h^0(D,\OO_D)\geq 2$. Then there is a decomposition $D=D_1+D_2$ where $D_1$, $D_2$ are curves such that
\begin{enumerate}[\rm (\rm i)]
\item  $h^0(D_1,\OO_{D_1}(-D_2))\neq 0$;

\item $D_1 D_2\leq 0$;
\item for each component $\Gamma$ of $D_1$, $\Gamma  D_2\leq 0$;
\item for each component $\Gamma$ of $D_1$, the restriction map $$H^0(D_1,\OO_{D_1}(-D_2))\to H^0(\Gamma,\OO_{\Gamma}(-D_2))$$ is injective;

\item  for each component $\Gamma$ of $D_1$, $$h^0(D,\OO_D)\leq h^0(D_2,\OO_{D_2})+h^0(\Gamma,\OO_{\Gamma}(-D_2)).$$
 \end{enumerate}
\smallskip 
Furthermore if $D_1D_2=0$ then $\OO_{D_1}(-D_2)=\OO_{D_1}$, and $h^0(D_1,\OO_{D_1})=1$.
\end{lem}
We note  the following immediate consequence of Lemma \ref{dec}:
\begin{cor}\label{cor: h0}  A $1$-connected curve $D$ satisfies $h^0(D,\OO_D)=1$. 
\end{cor}

\begin{remark} 
Lemma \ref{dec} means that a curve $D$ such that  $h^0(D,\OO_D) \geq 2 $  is  necessarily not $1$-connected. However  $h^0(D,\OO_D)=1$ does not imply 1-connectedness. 
For instance a multiple fibre  $F=mF'$ of a fibration with connected  fibres is not 1-connected  but  $h^0(F,\OO_F)=1$ (cf.~\cite[Chp. III]{bpv}).

If $D$ is reduced  then $h^0(D,\OO_D)$ is exactly the number of connected components of $D$.  On the other hand a general $D$ may have connected support and $h^0(D,\OO_D)> 1$. For instance if $E$ is an irreducible $(-1)$-curve and $D=2E$,  then  $h^0(D,\OO_D)=3$.

\end{remark}

 From  the above we see:

\begin{lem} \label{lem: pa} Let $D$ be a curve on a surface $S$. The following hold:
\begin{enumerate}
\item  $p_a(D)\leq h^0(D,\omega_D)$ and equality holds if and only if $h^0(D,\OO_D)=1$;

\item  for any curve $A<D$ one has  $h^0(A,\omega_A)\leq h^0(D,\omega_D)$;
\item  if $D$ is $1$-connected then  $p_a(D)=h^0(D,\omega_D)$;
\item if $D$ is $1$-connected then for any curve $A<D$, $p_a(A)\leq p_a(D)$;
\item if $q(S)=0$, then $h^0(D,\omega_D)= h^0(S, K_S+D)-p_g(S)$;
\item  the kernel of the natural restriction map $H^1(S, \OO_S)\to H^1(nD, \OO_{nD})$ is independent of $n$; 
\item let $d$ be the dimension of the kernel of the restriction map  $H^1(S, \OO_S)\to H^1(D, \OO_D)$. For any $n\in \mathbb N$, $$h^1(S, \OO_S(-nD))=h^0(nD, \OO_{nD})-1 
+d.$$\end{enumerate}

\end{lem}

\begin{proof}
(i) follows from the definition of $p_a(D)$ and Serre duality.

(ii) follows by twisting the decomposition sequence for  $D_1=A$ and $D_2:=D-A$ by $K_S+D$ and using adjunction. 

(iii) and (iv)  are consequences of (i), (ii)  and Corollary \ref{cor: h0}.

(v) follows by taking cohomology of the exact sequence:
$$0\to K_S\to K_S+D\to \omega_D\to 0.$$

 (vi)   is due to C.P. Ramanujam (\cite{ra})   (see also, e.g., \cite{bpv}, IV.12.7).

Finally (vii) follows by taking the cohomology sequence  of the exact sequence:
$$0\to \OO_S(-D)\to  \OO_S\to \OO_D\to 0$$ and using (vi).

\end{proof}
 \section{Exceptional divisors of blow-ups of surfaces}

In this section we consider a birational morphism $p\colon T\to W$ of smooth projective surfaces and we study  the numerical properties of the exceptional divisor of $p$. 

We start by recalling some  basic facts, also in order to set  the notation:  
\begin{lem}\label{lem: Ei} Let  $p\colon T\to W$ be a birational morphism of smooth projective surfaces and set $s:=K^2_W-K^2_T$. Then 
there are effective non-zero divisors  $E_1, \dots E_s$ of $T$ such that: 
\begin{enumerate}
\item $K_T=p^*K_W+ \sum_{i=1}^s E_i$ and  $p^*K_W\Gamma=0$ for every component $\Gamma$ of $\sum_{i=1}^s E_i$; 
\item the irreducible components of $E_i$ are smooth rational curves and the dual graph of the support of $E_i$ is a tree;
\item $K_TE_i=-1$ for all $i$,  $E_iE_j=-\delta_{ij}$
\item the intersection form on the  components of $\sum_{i=1}^s E_i$ is negative definite.
\end{enumerate}
\end{lem}
\begin{proof}  The morphism $p$ is a composition $\epsilon_s \circ \dots \circ\epsilon_1$ of blow ups at (possibly infinitely near) points  $q_1,\dots q_s$ of $W$ and, for $i=1,\dots s$, the curve $E_i$  is the total transform on $T$ of the exceptional curve of $\epsilon_i$. 
Now (i) and (iii) follow from the definition of $E_i$ and from the formula for the canonical class of a blow-up. (ii) is easily proven by induction on $n$.

(iv) is a general property of divisors that can be contracted (cf. for instance \S A.7 of \cite{r}), but also follows directly from (ii), since $E_1,\dots E_s$ span the same subgroup  of $\Pic(T)$ as the components of $\sum_{i=1}^sE_i$.
\end{proof}

Next we explore subtler numerical properties of the exceptional divisor:

\begin{prop} \label{prop: A} In the same assumptions and notation as in Lemma \ref{lem: Ei}, one has:
\begin{enumerate}
\item  each component  $\Gamma$ of $E_i$ satisfies $\Gamma E_i\leq 0$; 
\item  $E_i$ has a unique component  $\Gamma_i$  such that $\Gamma_i E_i<0$. In addition $\Gamma_i E_i=-1$  and $\Gamma_i$ appears with multiplicity 1 in $E_i$;

\item   $E_i$ is 1-connected for all $i$;

\item any effective non-zero divisor $A$  contracted by $p$   satisfies $h^0(A,\omega_A)=0$  (and hence $p_a(A)\leq 0$);

\item  if $E_i$ and $E_j$ with $i\neq j$ have a common component then either $E_j<E_i$ or $E_i<E_j$;
\item if $E_i<E_j$ with $i\neq j$, every component $\Gamma$ of $E_i$ satisfies $\Gamma E_j=0$;

\item  any effective divisor   $\Delta$  contracted by $p$  and satisfying $\Delta^2=-1$  is one of the $E_i$;

\item  any 1-connected effective divisor $\Delta$  contracted by $p$  satisfies $\Delta E_i\leq 1$ for each $i$.
\end{enumerate} 
\end{prop}

\begin{proof}
(i) is well known in greater generality (see \cite{r}, \S 4), and in our case it is easily proven as follows. Write  $p$ as a composition $\epsilon_s \circ \dots \circ\epsilon_1$ of blow ups at (possibly infinitely near) points  $q_1,\dots q_s$ of $W$ and let $p_i:=\epsilon_s \circ \dots \circ\epsilon_i\colon T\to W_i$. Take a smooth curve $C$ on $W_i$ passing through $q_i$ and not passing through $q_{i+1}, \dots q_s$, so that $p_i^*C=C'+E_i$, where $C'$ is the strict transform of C. Then  we have $0=\Gamma (p_i^*C)=\Gamma C'+\Gamma E_i$, namely $\Gamma E_i=-\Gamma C'\le 0$, since  $\Gamma$ is not a component of $C'$. 
\medskip

(ii)  Write  $E_i=\sum_jn_{ij}\Gamma_{ij}$, with $n_{ij}>0$ and $\Gamma_{ij}$ distinct irreducible curves.  Then $-1=E_i^2=\sum_jn_{ij}\Gamma_{ij}E_i$ and the claim follows immediately from (i).
\medskip

(iii) Let $E_i=A+B$, with $A,B>0$ be a decomposition. Note that  we have $A^2<0, B^2<0$  by Lemma \ref{lem: Ei} (iv). So from $-1=E_i^2=A^2+2AB+B^2$  one has necessarily $AB\geq 1$. 
\medskip

(iv) To prove  the assertion  it suffices to show that each connected component $\Gamma$ of $A$ satisfies $h^0(\Gamma,\omega_\Gamma)=0$. So by Lemma \ref{lem: pa} (ii) it is enough to show that, for each $E_i$   and every $n>0$, $h^0(nE_i, \omega_{nE_i})=0$.  

 Since by duality  $h^0(nE_i, \omega_{nE_i})=h^1(nE_i, \OO_{nE_i})$, it suffices to show that  $h^1(nE_i, \OO_{nE_i})=0$.

 The   exact sequence
$$0\to \OO_T(-nE_i)\to \OO_T\to \OO_{nE_i}\to 0$$ induces the cohomology sequence

\begin{gather*}\dots\to H^1(T, -nE_i)\to H^1(T, \OO_T)\to H^1(nE_i, \OO_{nE_i})\to \\
\to H^2(T,-nE_i)\to H^2(T, \OO_T)\to 0.
\end{gather*}

Since $h^0(T, K_T)=h^0(T,K_T+nE_i)$, and thus by  duality  $h^2(T, \OO_T)=h^2(T,-nE_i)$,  the map $H^1(T, \OO_T)\to H^1(nE_i, \OO_{nE_i})$ is surjective.

Now Lemma \ref{lem: Ei} (iii)  and the adjunction formula give $p_a(E_i)=0$. Since $E_i$ is 1-connected by (iii), we have $h^0(E_i, \omega_{E_i})=0$ and so by duality  $h^1(E_i, \OO_{E_i})=0$. 
Hence the kernel of the map  $ H^1(T, \OO_T)\to H^1(E_i, \OO_{E_i})$ is exactly  $ H^1(T, \OO_T)$.  Then, by Lemma \ref{lem: pa} (vi), also  the kernel of the map  $ H^1(T, \OO_T)\to H^1(nE_i,\OO_{nE_i})$ is exactly  $ H^1(T, \OO_T)$  implying $h^1(nE_i,\OO_{nE_i})=0$.
\medskip

(v) Write $E_i=A+B$, $E_j=A+C$ where $A,B,C$ are effective divisors with $A\neq 0$ and $B$ and $C$ have no common components (and so $BC\geq 0$).  If $B=0$ or $C=0$ the claim is proven, so assume by contradiction that both $B$ and $C$ are non-zero. From   $-2=(E_i-E_j)^2=(B-C)^2=B^2-2BC+C^2$ and Lemma \ref{lem: Ei} (iv) we have necessarily $BC=0$ and   $B^2=C^2=-1$. 

 From (iii) we have $AB>0, AC>0$.  Since from (i)  $BE_i\leq 0$, $CE_j\leq 0$,  the only possibility  is $AB=AC=1$ and therefore $A^2=E_i^2-2AB-B^2=-2$.  But then 
$(A+B+C)^2= 0$, contradicting Lemma \ref{lem: Ei}, (iv).  
\medskip

 (vi) It suffices to note that  $E_iE_j=0$ and that $E_i<E_j$ implies by (i) that every component $\Gamma$ of $E_i$ satisfies $\Gamma E_j\le 0$. 

\medskip

(vii) Since $\Delta$ is contracted by $p$, arguing as in the proof of (iii)  we see that $\Delta$ is 1-connected and so  $p_a(\Delta)=h^0(\omega_{\Delta})=0$ by (iv).   Hence $K_T\Delta=-1$. From $(p^*K_W)\Delta=0$ we conclude that $\Delta \sum E_i=-1$.  So there is at least one $E_j$ such that $E_j\Delta<0$. But then $(E_j-\Delta)^2\geq 0$ and so $E_j-\Delta=0$ by Lemma \ref{lem: Ei}, (iv). 
\medskip

(viii) By (iv) we have $p_a(\Delta+E_i)\le 0$ and $p_a(\Delta)=p_a(E_i)=0$, since the latter two divisors are $1$-connected. Since $p_a(\Delta+E_i)=p_a(\Delta)+p_a(E_i)+\Delta  E_i-1$ it follows   $ \Delta E_i\leq 1$.

\end{proof}

 \section{Exceptional divisors with $K\Delta=0$ and $\Delta^2=-2$}
 
Our focus in this section is on the properties of effective divisors  $\Delta$ with $K_T\Delta=0$ and $\Delta^2=-2$ that are contracted by a birational morphism of smooth surfaces. The results  obtained here will be applied in the next section to A-D-E configurations (cf. Definition \ref{def: ADE}).

\begin{prop}\label{prop: delta}
 In the same assumptions and notation as in Lemma \ref{lem: Ei}, let $\Delta$ an  effective divisor of $T$ such that $K_T\Delta=0$, $\Delta^2=-2$. 
\newline If $\Delta$ is contracted by $p$, then: 
\begin{enumerate}

\item   $\Delta$ is 1-connected; 
\item  if $\Delta'\ne \Delta$ is also an effective divisor contracted by $p$ and such that $K_T\Delta'=0$, ${\Delta'}^2=-2$, then $-1\le \Delta\Delta'\le 1$;

\item   there are indices $j$ and $k$ such that:
$$\Delta E_j=1, \quad \Delta E_k=-1,\quad \Delta E_i=0 \quad \mbox{for} \ i\ne j,k$$
and   $E_k=E_j+\Delta$.

\item  Let  $E_j$ be such that $\Delta E_j=1$. If $\Delta'\ne \Delta$ is also an effective divisor contracted by $p$ and such that $K_T\Delta'=0$, ${\Delta'}^2=-2$ and  $\Delta' \Delta=0$, then $\Delta' E_j=0$.

\end{enumerate}
\end{prop}

\begin{proof}
(i) If   $\Delta=A+B$ with $A,B>0$, then   $-2=\Delta^2= A^2+2AB+B^2$. If   $AB\leq 0$, then  by Lemma \ref{lem: Ei} (iv)  the only possibility is $A^2=B^2=-1$.  Now $K_TA$ and $K_TB$ are non zero  by the adjunction formula, so  from $K_T\Delta=0$ we see that, say, $K_TA>0$ contradicting $p_a(A)\leq 0$ (cf. Proposition \ref{prop: A}, (iv)).
\medskip

(ii) Again by Lemma \ref{lem: Ei} (iv), we have $0>(\Delta+\Delta')^2=-4+2\Delta \Delta'$, namely $\Delta \Delta'\le 1$. Similarly, $0>(\Delta-\Delta')^2$ gives $\Delta \Delta'\ge -1$.
\medskip

(iii) 
 Let $L$ be the subgroup of $\Pic(T)$ generated by the irreducible curves contracted by $p$. The curves $E_1,\dots E_s$ give a basis for $L$ and the intersection form restricted to $L$ is negative definite, so there is at least an index $j$ with $\Delta E_j\ne 0$.  Since $\Delta p^*K_W=0$, we have $0=\Delta K_T=\Delta\sum E_i$, so we may assume $\Delta E_j>0$. By Proposition \ref{prop: A} (viii) it follows  $\Delta E_j=1$. So the effective divisor $\Delta+E_j$ satisfies $(\Delta+E_j)^2=-1$,  and therefore  by Proposition \ref{prop: A} (vii) there is an index $k$ such that $E_k=E_j+\Delta$. We have $-1= E_k^2=E_kE_j+E_k\Delta=E_k\Delta$. If $i\ne j, k$ then we have $0=E_kE_i=E_jE_i+\Delta E_i=\Delta E_i$.

  \medskip
 (iv)  Consider the divisors $D= 2E_j+\Delta+\Delta'$ and  $D'= 2E_j+\Delta-\Delta'$, which are non zero because $K_TD=K_TD'=-2$.  Then $D^2=4E_j^2+4E_j \Delta+4E_j \Delta'+\Delta^2+\Delta'^2=4E_j\Delta'-4$ and  $D'^2=-4E_j\Delta'-4$. By negative definiteness  (Lemma \ref{lem: Ei} (iv)) the only possibility is $E_j\Delta'=0$.

  \end{proof}

We apply Proposition \ref{prop: delta}  to characterize irreducible $p$-exceptional curves with intersection $\le -3$. Note that the analogous statement for a $p$-exceptional $-2$-curve follows by Proposition \ref{prop: delta} (iii). 

\begin{prop} In the same assumptions and notation as in Lemma \ref{lem: Ei}, let $N$ be an irreducible curve  such that   $N^2=-m\leq - 3$; set $J=\{j\ |\ NE_j=1\}$.  

If $N$ is contracted by $p$ then:
\begin{enumerate}
\item $|J|=m-1$;
\item    $\Gamma:=N+\sum_{j\in J}E_j$   is one of the $E_i$;
\item  $\Gamma$ contains $N$ with multiplicity 1.
\end{enumerate}
\end{prop}
\begin{proof} 
The curve $N$ is smooth rational by Lemma \ref{lem: Ei} (ii).   
Since $N^2=-m$,    we have $(\sum_{i=1}^s E_i)N=K_TN=m-2$ by adjunction. Since $NE_i\leq 1$ for all $i$  by  Proposition  \ref{prop: A} (viii), we have $|J|\ge (m-2)$. 
Pick $j_1< j_2<\dots <j_{m-2}\in J$ and set  $\Delta=N+\sum_{i=1}^{m-2} E_{j_i}$. One has  $\Delta^2= -2$, $K_T\Delta=0$ , so by Proposition \ref{prop: delta} (iii) there are  (distinct) indices  $j_{m-1}$ and $j_m$ with $E_{j_{m-1}} \Delta =1$ 
and  $E_{j_m}=\Delta +E_{j_{m-1}}$. Note that $E_{j_1}\Delta =\dots=E_{j_{m-2}}\Delta=0$, so that $j_{m-1}$  and $j_m$ are  distinct from $j_1,\dots j_{m-2}$. 
Since $E_{j_m}N=-1$ and $N E_k=0$ for $k\ne j_1, \dots j_m$, we have  $J=\{j_1,\dots j_{m-1}\}$.
\end{proof}

 \section{A-D-E configurations}\label{ADE}
 
 In this section we apply the results of the previous section to A-D-E configurations, that we now  introduce: 

\begin{defn}\label{def: ADE} An {\em  A-D-E configuration} of curves on a smooth projective surface $T$ is a reduced effective divisor $\Delta$ whose  components are $-2$-curves  and whose dual graph is a Dynkin diagram of type $A_n$ ($n\ge 1$), $D_n$ ($n\ge 4$)  or $E_n$ ($n=6,7,8$).  Note that one has  $K_T\Delta=0$ by assumption and that it is easy to show $\Delta^2=-2$ by induction on the number $n$ of components of $\Delta$.
\end{defn}

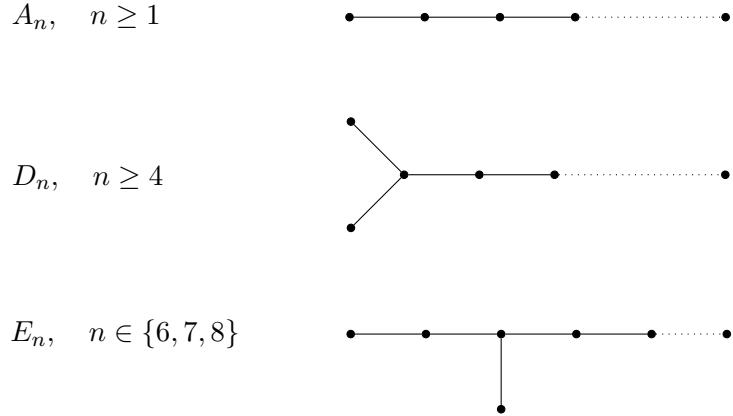
\begin{figure} [h!]
 \begin{tikzpicture}[main/.style = {draw,circle,fill=black, inner sep=1pt}]
 \node[main] (1) {};
 \node[main] (2) [right of=1] {};
 \node[main] (3) [right of=2] {};
 \node[main] (4) [right of=3] {};
 \node[main] (5) [right of=4, node distance=2cm] {};
 \draw (1) -- (2) node at (-3.5,0) {$A_n,\quad  n\geq 1$ } ;
  \draw (2) -- (3) ;
  \draw (3) -- (4) ;
  \draw[dotted] (4) -- (5);
\end{tikzpicture}

\vspace{1cm}

 \begin{tikzpicture}[main/.style = {draw,circle,fill=black, inner sep=1pt}]
 \node[main] (1) {};
 \node[main] (2) [below right of=1] {};
 \node[main] (3) [below left of=2] {};
 \node[main] (4) [right of=2] {};
 \node[main] (5) [right of=4] {};
 \node[main] (6) [right of=5, node distance=2.27cm] {};
 \draw (1) -- (2) node at (-3.5,-0.75) {$D_n,\quad n\geq 4$ };
  \draw (2) -- (3) ;
  \draw (2) -- (4);
  \draw (4) -- (5);
   \draw[dotted] (5) -- (6);
\end{tikzpicture}

\vspace{1cm}

 \begin{tikzpicture}[main/.style = {draw,circle,fill=black, inner sep=1pt}]
 \node[main] (1) {};
 \node[main] (2) [right of=1] {};
 \node[main] (3) [right of=2] {};
 \node[main] (4) [right of=3] {};
 \node[main] (5) [right of=4] {};
 \node[main] (6) [right of=5] {};
 \node[main] (7) [below of=3] {};
  \draw (1) -- (2) node at (-3,0) {$E_n,\quad n\in\{6,7,8\}$} ;
  \draw (2) -- (3) ;
  \draw (3) -- (4) ;
  \draw (4) -- (5) ;
  \draw[dotted] (5) -- (6);
   \draw (3) -- (7);
\end{tikzpicture}
 \caption{Dynkin diagrams.} \label{ADEgraphs}
 \end{figure}

\bigskip
It is  well known (cf. \cite{bpv}, III.3) that A-D-E configurations arise as exceptional divisors in the minimal resolution of canonical surface singularities; they are characterized  by the condition that  they are reduced connected  exceptional divisors such that  for any of their irreducible components $\Gamma$ one has  $K_T\Gamma=0$ (cf. \cite{bpv}, Prop. III.2.4).
 The  numerical conditions $K_T\Delta=0$ and $\Delta^2=-2$ are far from giving a characterization of A-D-E configurations, as shown by the example below:

\begin{ex} 
(a) The fundamental cycle of a canonical singularity $(X,p)$  (cf. \cite{bpv}, III.3) is a divisor $Z$ whose support is the exceptional divisor  of the minimal resolution $T\to  X$ and satisfies $K_TZ=0$ and $Z^2=-2$, but $Z$ is reduced only if $p$ is a singularity of type $A_n$.

(b) It is easy to produce divisors with $K_T\Delta=0$ and $\Delta^2=-2$ such that not all the components of $\Delta$ are $-2$-curves: just take an A-D-E configuration, blow it up at a  point and take its total transform on the blow up. 
\end{ex}
   
  The next result appears in \cite{spiva}, p.~426. 
  \begin{prop} \label{thm: An} Let  $p\colon  T\to W$ be a birational morphism of smooth surfaces and let $\Delta$ be an A-D-E configuration of $T$.
  
  If $\Delta$ is contracted by $p$ then it is of type $A_n$.
    \end{prop}

  \begin{proof} We use the notation for the $p$-exceptional divisors introduced in Lemma \ref{lem: Ei}.
  
   Since every A-D-E configuration not of type $A_n$ contains a subconfiguration of type $D_4$, it suffices to rule out $D_4$. So assume by contradiction  that $\Delta$ is of type $D_4$ and denote 
 by $N_0$ the central component and by $N_1, N_2,N_3$ the remaining ones. By Proposition \ref{prop: A} (iii) there is precisely one of the $E_i$ that has intersection $1$ with $N_0$. Call it $E_0$. Now note that at most one of the $N_i$ with $i>0$ satisfies $N_iE_0=-1$ (because  by Proposition \ref{prop: A} (ii) there is a unique component of $E_0$ that intersects $E_0$ negatively).   Then, say, $N_1 E_0\geq 0$, $N_2 E_0\geq 0$  and so $N_1(E_0+N_0)\geq 1 $,  $ N_2(E_0+N_0)\geq 1$.   Then $(2( E_0+N_0)+N_1+N_2)^2\ge 0$ and this is a contradiction to negative definiteness (Lemma \ref{lem: Ei} (iv)). 
 
\end{proof}

 \begin{prop}\label {prop: theta}Let  $p\colon  T\to W$ be a birational morphism of smooth surfaces and $\Delta, \Delta'$ be A-D-E configurations of type $A_n$   contracted by $p$.  Then (keeping the notation for the $p$-exceptional divisors introduced in Lemma \ref{lem: Ei}):
  \begin{enumerate}

  \item  if $\Delta E_j=1$ then $\Delta$ and $E_j$ have no common components (and so intersect transversely in one point);
  
  \item   if $\Delta E_j=1$ then the unique component $\theta_{\Delta}$ of $E_j$ such  that $\theta_{\Delta} \Delta=1$ is also the unique component of $E_j$ such that  $\theta_{\Delta}E_j=-1$ (see Proposition \ref{prop: A} (ii));
 
  \item  if $\Delta E_j=1$ and $\Delta\cap \Delta'=\emptyset$  then  either $\Delta'$ is disjoint from $E_j$ or $\Delta'<E_j$;

  \item  if $\Delta\cap \Delta'=\emptyset$  then $\theta_{\Delta}\neq \theta_{\Delta'}$.
\end{enumerate}

  \end{prop}

    \begin{proof}
    We use again the notation for the $p$-exceptional divisors introduced in Lemma \ref{lem: Ei} and we recall  that by Proposition \ref{prop: delta} if $\Delta E_j=1$ then $E_k:=\Delta+E_j$ is one of the $E_i$.

 (i)   Write $E_j=A+B$, $\Delta=A+C$ where $A,B,C$ are effective divisors such that $B$ and $C$ have no common components (and so $BC\geq 0$). 
 
 Suppose that $A\neq 0$. Since $AE_j\leq 0$ by Proposition \ref{prop: A} (i)  and $\Delta E_j=1$, $C\neq 0$. On  the other hand since $K_TE_j=-1$ and by hypothesis $K_TA=0 $, $B$ is not the zero divisor.
 Since  $E_k=E_j+\Delta=2A+B+C$ is one of the $E_i$,  one has   $BE_k\leq 0$, by Proposition \ref{prop: A} (i).  Then $BC\geq 0$ implies $B^2+2AB\leq 0$.  From $-1=E_j^2=A^2+2AB+B^2$ we obtain $A^2\geq -1$, which is impossible by the adjunction formula  because $A^2<0$ and $K_TA=0$.  So $A=0$.
 
  Since $\Delta E_j=1$ and $\Delta$ and $E_j$ have no common components  they intersect transversely in one point. 
 \medskip 
 
(ii)  Let $\theta_{\Delta} $ be the unique component of $E_j$ such that $\theta_{\Delta} \Delta=1$. Consider as above $E_k=\Delta+E_j$. Since   $\theta_{\Delta} E_k\leq 0$  (cf. Proposition \ref{prop: A} (i)) and  $\theta_{\Delta}\Delta=1$, necessarily $\theta_{\Delta}E_j=-1$ (cf. Proposition \ref{prop: A} (ii)).
\medskip
 
 (iii)  Write $E_j=A+B$, $\Delta'=A+C$ where $A,B,C$ are effective divisors such that $B$ and $C$ have no common components (and so $BC\geq 0$).  Note that $B\neq 0$  and $B^2$ is odd because $KE_j=-1$, whilst by hypothesis every component  $\Gamma$ of $\Delta'$ satisfies $K_T\Gamma=0$. 
  Note also that  because $\Delta$ and $\Delta'$ are disjoint the unique component $\theta_{\Delta} $  of $E_j$ such that $\theta_{\Delta} \Delta=1$ is a component of $B$ and therefore $BE_j=-1$, by (ii)  and  by Proposition \ref{prop: A}. 
    
     If  $A=0$  then $E_j$ and $\Delta'$ are disjoint, since $\Delta'E_j=0$ by Proposition \ref{prop: delta} (iv). On the other hand  if $C=0$, $\Delta'<E_j$.  
     
     So assume $A\neq 0$ and $C\neq 0$.   Then $-3=(E_j-\Delta')^2=(B-C)^2=B^2+C^2-2BC$.  By negative definiteness  $B^2<0$, $C^2<0$, and  because $B^2$ is odd the only possibility is $BC=0$, $B^2=-1$ and $C^2=-2$.  But then, since  by Proposition \ref{prop: A} (iii),  $E_j$ is 1-connected $BA\geq 1$ implies $BE_j\ge 0$ contradicting $BE_j=-1$.

 \medskip
 (iv)  By (iii) either $\Delta'$ is disjoint from $E_j$ in which case $\theta_{\Delta}\neq \theta_{\Delta'}$ by the definition of $\theta_{\Delta'}$, or $\Delta'<E_j$.  In the latter case let $E_h$ be the unique $E_i$ such that $\Delta' E_i=1$. Note that by Proposition \ref{prop: delta} (iv), $E_h\neq E_j$.
 
  Suppose that $\theta_{\Delta}=\theta_{\Delta'}$. Then $E_j$ and $E_h$ have common components and so by   Proposition \ref{prop: A} (v) either $E_j<E_h$ or $E_h<E_j$.  Since  by (ii) $\theta_{\Delta}E_j=\theta_{\Delta}E_h=-1$ we have a contradiction  to   Proposition \ref{prop: A} (vi).
   \end{proof}

 \begin{cor} \label{cor: number}  Let  $p\colon  T\to W$ be a birational morphism of smooth surfaces and let  $\Delta_{1},....,  \Delta_{k}$  be  $k$ disjoint A-D-E configurations of type $A_{m_i}$  contracted by $p$.   Then   $\sum_{i=1}^k (m_i+1)\leq s$  where  $s:=K_W^2-K_T^2 $ is the number of blowups from $W$ to $T$.

In particular the number $k$ of disjoint irreducible $-2$-curves (i.e  A-D-E configurations of type $A_1$) contracted by $p$ satisfies $2k\leq s$.
\end{cor}
 
  \begin{proof}   For each $\Delta_i$, $i=1,...,k$  let  $\theta_{\Delta_i}$ be as in Proposition \ref{prop: theta} (ii). Since each $\Delta_i$ has exactly $m_i$ components   and the curves $\theta_{\Delta_i}$  are all distinct  the exceptional locus of $p$ has at least $\sum_{i=1}^k (m_i+1)$ irreducible components and so  $\sum_{i=1}^k (m_i+1)\leq s$.   \end{proof}

 \medskip
 
 \begin{prop} \label{prop: nef} Let  $p\colon  T\to W$ be a birational morphism of smooth surfaces and let $\Delta$ be an A-D-E configuration on $T$.  
 If $K_W$ is nef,  then one of the following holds:
 \begin{itemize}
 \item[\rm (a)] $\Delta$ is mapped to a point by $p$; 
 \item[\rm (b)] $\Delta$ is disjoint from all the $p$-exceptional curves. 
 \end{itemize}
   \end{prop}
\begin{proof}
We keep using the notation for the $p$-exceptional divisors introduced in Lemma \ref{lem: Ei}.  

Let $N$ be a $-2$-curve. If $N$ is not a component of $\sum_{i=1}^s E_i$, i.e., if $N$ is not contracted by $p$, then $N\sum_{i=1}^s E_i\geq 0$. 
Since $K_W$ is nef,  the equality  $0=K_TN=(p^*K_W)N+(\sum_{i=1}^s E_i) N$ implies $(\sum_{i=1}^s E_i) N=0$ and so  $N$ is disjoint from the support of $\sum_{i=1}^s E_i$. Since the support of $\sum_{i=1}^s E_i$ contains all $p$-exceptional curves, the claim is proven in this case.

In general, given an A-D-E configuration $\Delta$, we can write it as $\Delta_1+\Delta_2$, where $\Delta_1$ is the sum of all components of $\Delta$ that are contracted by $p$ and $\Delta_2=\Delta-\Delta_1$. The divisors $\Delta_1$ and $\Delta_2$ are disjoint  by the first part of the proof, but $\Delta$ is connected, so either $\Delta=\Delta_1$ (and we are in case (a)) or $\Delta=\Delta_2$ (and we are in case (b)). 
\end{proof}
\begin{remark}   If in $\mathbb P^2$ we blow up three points on a line and then blow up  the three exceptional divisors in one point each we obtain a $D_4$ configuration with common components with $\sum_{i=1}^s E_i$ but not contained in $\sum_{i=1}^s E_i$. So Proposition  \ref{prop: nef} fails if $K_W$ is not nef. 
\end{remark}

\section{Surfaces with canonical singularities}

Let $(Y,q)$ be a (germ of) canonical surface singularity and let $\Delta_q$ be the exceptional divisor of the minimal resolution; $\Delta_q$ is an A-D-E configuration. It is well known that $(Y,q)$ is analytically isomorphic to the quotient singularity $(\C^2/G,0)$, where $G$ is a finite subgroup of $SL(2,\C)$. Since $SL(2,\C)$ acts freely on $\C^2\setminus\{0\}$, $G$ is isomorphic to the local fundamental group of $(Y,q)$. Following \cite{miyaoka}, we attach to $(Y,q)$ the following numerical invariant:
$$\nu(q)=e(\Delta_q)- \frac{1}{|G|}.$$
Since the dual graph of $\Delta_q$ is a tree and all the irreducible components of $\Delta_q$ are isomorphic to $\mathbb P^1$ the Euler  number $e(\Delta_q)$ is $n+1$,where $n$ is the number of components of $\Delta_q$.  One has the following (\cite{br}, also \cite{bu}, \cite {r2}): 

\bigskip
\renewcommand{\arraystretch}{1.5}
\begin{tabular}{|l|c|c|c|}
    \hline
    A-D-E & $G$&$ |G| $ & $ \nu(q) $\\
    \hline
   $ A_n$ &cyclic& $n+1$  & $n+1 - \frac{1}{n+1}$ \\
    \hline
    $D_n$ &binary dihedral&$4(n-2)$  & $n+1 - \frac{1}{4(n-2)}$ \\
    \hline
   $E_6$ &binary tetrahedral & $24$ & $7- \frac{1}{24}$\\
     \hline
    $E_7$ &binary octahedral group&  $48$ & $8- \frac{1}{48}$\\
     \hline
    $E_8$ &binary icosahedral group & $120$ & $9- \frac{1}{120}$\\
      \hline
\end{tabular}
\renewcommand{\arraystretch}{1}

\bigskip

Let $X$ be a normal  surface with canonical singularities $q_1,\dots q_h$ and smooth elsewhere;  if $K_X$ is nef (and hence $\kappa(X)\ge 0$),  then by Corollary 1.3 of \cite{miyaoka} (see also \cite{hi}) one has
\begin{equation}\label{eq: ineq}
\sum_{i=1}^h \nu(q_i) \le 12\chi(\OO_X)-\frac 4 3K_X^2.
\end{equation}

 Using the results of \S \ref{ADE}, we extend  the result  to the case  $\kappa(X)\ge 0$ (and $K_X$ possibly not nef):
\begin{thm}\label{thm: ineq}
Let $X$ be a surface with canonical singularities $q_1,\dots q_h$  and smooth elsewhere.  Let $f\colon T\to X$ be the minimal resolution  and  let $p\colon T\to W$ be the morphism to the minimal model.

 If $\kappa(X)\ge 0$, then
$$ \sum_{i=1}^h\nu(q_i)\le 12\chi(\OO_X)-\frac 4 3K_X^2-\frac{s}{3}$$ where $s=K^2_W-K^2_T$  is the number of blow ups from $W$ to $T$.
  
  In particular, if \ $\sum_{i=1}^h\nu(q_i)= 12\chi(\OO_X)-\frac 4 3K_X^2$, then $K_X$ is nef. 
 \end{thm} 
 \begin{proof}
 For $i=1,\dots h$ let $\Delta_i$  be the exceptional divisor mapping to $q_i$ via $f$.  Then  $\Delta_1,\dots \Delta_h$ are disjoint   A-D-E configurations.  Since $\kappa(X)\ge 0$,  $K_W$ is nef and so  by Proposition \ref{prop: nef}  each $\Delta_i$   is either

 \begin{itemize}
 \item[\rm (a)] mapped to a point by $p$ or
 \item[\rm (b)]  mapped to an A-D-E configuration of $W$. 
 \end{itemize}

 Let $\Delta_1,...,\Delta_k$ be the  A-D-E configurations  in case (a) and $\Delta_{k+1}, ...,\Delta_h$ be the  A-D-E configurations  in case (b) .

 By  Miyaoka's inequality    \eqref{eq: ineq}   $$  \sum_{i=k+1}^h\nu(q_i)\le 12\chi(\OO_X)-\frac 4 3K_W^2. $$
 
 Since $K_W^2=K_T^2+s$ and $K_X^2=K_T^2$ the above inequality can be written $$  \sum_{i=k+1}^h\nu(q_i)\le 12\chi(\OO_X)-\frac 4 3K_X^2 -\frac 4 3s. $$
 
 On the other hand  by  Proposition \ref{thm: An}  each $\Delta_i$, $i=1,...,k$  is of type $A_{m_i}$. Since $\nu(q_i)=m_i+1-\frac{1}{m_i+1}$,   by Corollary \ref{cor: number} one has $$  \sum_{i=1}^k\nu(q_i)\le s-\sum_{i=1}^k\frac{1}{m_i+1}.$$

  Hence    $$  \sum_{i=1}^h\nu(q_i)\le 12\chi(\OO_X)-\frac 4 3K_X^2 -\frac 4 3s +s-\sum_{i=1}^k\frac{1}{m_i+1}\leq 12\chi(\OO_X)-\frac 4 3K_X^2 -\frac s 3.$$  
\end{proof}

\begin{remark} \label{langer}

 Generalizations  of Miyaoka's inequality    \eqref{eq: ineq}  for log canonical surface pairs  were studied first by Megyesi and then  by Langer.   For the case we consider in Theorem \ref{thm: ineq},  and using the notation therein, their results   (\cite[Theorem 0.1 (i)]{megyesi}, \cite [Corollary 0.2]{langer1}, \cite [Corollary 5.2]{langer})  all yield  the inequality: 
 
  $$  \sum_{i=1}^h\nu(q_i)\le 12\chi(\OO_X)-\frac 4 3K_X^2 -\frac {s}{12},$$
 weaker than Theorem \ref{thm: ineq}.

\end{remark}

\bigskip

\bigskip

\begin{minipage}{13.0cm}
\parbox[t]{6.5cm}{Vicente Lorenzo\\
Departamento de  Matem\'atica\\
Instituto Superior de Economia\\ e Gest\~ao\\
Universidade  de Lisboa\\
R. do Quelhas\\
1200-781 Lisboa, Portugal\\
vlorenzogarcia@gmail.com
 } \hfill
\parbox[t]{5.5cm}{Margarida Mendes Lopes\\
Departamento de  Matem\'atica\\
Instituto Superior T\'ecnico\\
Universidade de Lisboa\\
Av.~Rovisco Pais\\
1049-001 Lisboa, Portugal\\
mmendeslopes@tecnico.ulisboa.pt
}

\vskip1.0truecm

\parbox[t]{5.5cm}{Rita Pardini\\
Dipartimento di Matematica\\
Universit\`a di Pisa\\
Largo B. Pontecorvo, 5\\
56127 Pisa, Italy\\
rita.pardini@unipi.it}
\end{minipage}

\end{document}